\documentclass[11pt]{amsart}

\input{xypic}

\newtheorem{theorem}{Theorem}[section]
\newtheorem{lemma}[theorem]{Lemma}
\newtheorem{corollary}[theorem]{Corollary}
\newtheorem{proposition}[theorem]{Proposition}
\theoremstyle{remark}
\newtheorem{remark}[theorem]{Remark}
\theoremstyle{remark}
\newtheorem{example}[theorem]{Example}

\newcommand\Ab{{{\mathcal A}b}} 
\newcommand\Md{{\text{\rm Mod}}}
\newcommand\Pm{{\text{\rm PMod}}}
\newcommand\Sm{{\text{\rm SMod}}}

\newcommand\Set{{{\mathcal S}et}}

\newcommand{\Hom}{\mathop{\rm Hom}\nolimits}

\newcommand{\Ext}{\mathop{\rm Ext}\nolimits}

\newcommand{\Ker}{\mathop{\rm Ker}\nolimits}

\newcommand{\ima}{\mathop{\rm im}\nolimits}

\newcommand\opp{{^{\text{\rm op}}}} 

\newcommand\CA{{\mathcal A}}
\newcommand\CB{{\mathcal B}}
\newcommand\CC{{\mathcal C}}

\newcommand\FF{{\mathfrak F}}
\newcommand\CG{{\mathcal G}}

\newcommand\CL{{\mathcal L}}

\newcommand\CO{{\mathcal O}}

\newcommand\CT{{\mathcal T}}

\newcommand\CU{{\mathcal U}}
\newcommand\CV{{\mathcal V}}

\newcommand\id{{\bf 1}}

\begin{document}

\title{Generalized lax epimorphisms in the additive case}

\author{George Ciprian Modoi}

\address{"Babe\c s-Bolyai" University, Faculty of Mathematics and Computer Science,
Chair of Algebra, 1, M. Kog\u alniceanu, RO-400084, Cluj-Napoca
\\ Romania}


\email[George Ciprian Modoi]{cmodoi@math.ubbcluj.ro}


\begin{abstract}
In this paper we call generalized lax epimorphism a functor
defined on a ring with several objects, with values in an abelian
AB5 category, for which the associated restriction functor is
fully faithful. We characterize such a functor with the help of a
conditioned right cancellation of another, constructed in a
canonical way from the initial one. As consequences we deduce a
characterization of functors inducing an abelian localization and
also a necessary and sufficient condition for a morphism of rings
with several objects to induce an equivalence at the level of two
localizations of the respective module categories.
\end{abstract}
\maketitle

\section*{Introduction}

All categories which we deal with are preadditive, i.e. there
exists an abelian group structure on the hom sets, such that the
composition of the morphisms is bilinear. For a category $\CC$ we
denote by $\CC(-,-):\CC^\opp\times\CC\to\Ab$ the bifunctor
assigning to every pair of objects the abelian group of all maps
between them. All functors between preadditive categories are
additive i.e. preserve the addition of maps. Consider a small
preadditive category $\CU$. Recall that a preadditive category
with exactly one object is nothing but an ordinary ring with
identity, therefore small preadditive categories are also called
{\em rings with several objects}. As in the case of ordinary
rings, a {\em(right) module over} $\CU$ (or simply, an
{\em$\CU$-module}) is functor $\CU^\opp\to\Ab$. All $\CU$-modules
together with natural transformations between them form an
abelian, AB5 category denoted $\Md(\CU)$, where limits and
colimits are computed point--wise. Moreover the Yoneda functor
\[\CU\to\Md(\CU),\hbox{ given by }U\mapsto\CU(-,U)\]
is an embedding and its image form a set of (small, projective)
generators for $\Md(\CU)$, therefore $\Md(\CU)$ is a Grothendieck
category. This embedding allows us to identify an object $U\in\CU$
with its image in $\Md(\CU)$, that is with the functor $\CU(-,U)$.
In the sequel we use freely this identification. We denote by
$\Hom_\CU(X,Y)$ the set of all $\CU$-linear maps (i.e. natural
transformations) between the $\CU$-modules $X$ and $Y$; that is
$\Hom_\CU(X,Y)=\Md(\CU)(X,Y)$.

Following \cite{ABS}, a functor between small non-additive
categories $T:\CU\to\CV$ is called a {\em lax epimorphism},
provided that the functor
\[T_*:[\CV\opp,\Set]\to[\CU\opp,\Set],\ T_*X=X\circ T\] is fully faithful
(Here $[\CU\opp,\Set]$ denotes the category of all contravariant
functors from $\CU$ to the category of sets). We shall use the
same terminology in the additive case (consequently, replacing
$[\CU\opp,\Set]$ with $\Md(\CU)$). We consider now a functor
$T:\CU\to\CC$, where $\CU$ a ring with several objects and $\CC$
is any cocomplete, abelian category. Then there is a unique, up to
a natural isomorphism, colimit preserving functor
$T^*:\Md(\CU)\to\CC$ such that $T^*\CU(-,U)=TU$, for all
$U\in\CU$. The functor $T^*$ has a right adjoint, namely the
functor
\[T_*:\CC\to\Md(\CU),\ T_*C=\CC(T-,C)\hbox{ for all }C\in\CC.\]
The functors $T^*$ and $T_*$ will be called the {\em induction},
respectively the {\em restriction functor} associated to $T$, and
the adjoint pair $(T^*,T_*)$ is said to be {\em induced by $T$}.
In accord with the above terminology, we call the functor $T$ {\em
generalized lax epimorphism}, if the associated restriction
functor $T_*$ is fully faithful.

For an additive functor $F$, we denote by $\Ker F$ the full
subcategory of the domain of $F$, consisting of all objects which
are annihilated by $F$, in contrast with $\ker$ which denotes the
categorical notion of kernel.

By an {\em abelian localization} we understood a pair of adjoint
functors between two abelian categories, with the properties that
the left adjoint is exact and the right adjoint is fully faithful.

Let $\CG$ be a ring with several objects. Recall that a {\em
localizing subcategory} in $\Md(\CG)$ is a full subcategory closed
under subobjects, quotients, direct sums and extensions.
Obviously, $\Ker F$ is a localizing subcategory, provided that $F$
is an exact, colimit preserving functor. It is well--known, that a
localizing subcategory is nothing but a hereditary torsion class,
so modules belonging to such a subcategory are called sometimes
{\em torsion modules}.

Consider a localizing subcategory $\CL$ in $\Md(\CG)$. We call
{\em $\CL$-torsion free} ({\em $\CL$-closed}) an object
$X\in\Md(\CG)$ satisfying $\Hom_\CG(L,X)=0$ (respectively
$\Hom_\CG(L,X)=\Ext_\CG^1(L,X)=0$) for all $L\in\CL$, where
$\Ext_\CG^1$ denotes as usually the first derived functor of
$\Hom_\CG$. We construct, as in \cite{G}, the quotient category
$\CC=\Md(\CG)/\CL$ together with the canonical (exact) functor
$Q:\Md(\CG)\to\CC$, called also the quotient functor, which has a
fully faithful right adjoint $R:\CC\to\Md(\CG)$. Clearly the pair
$(Q,R)$ is an abelian localization. Then $\CC$ is a Grothendieck
category, $\CL=\Ker Q$ and $R$ identifies $\CC$ with the full
subcategory of $\Md(\CG)$ consisting of all $\CL$-closed modules
(also see \cite{G}). Note also that $Q$ sends every morphism with
torsion kernel and cokernel in $\Md(\CG)$ into an isomorphism in
$\CC$, and is universal with this property. In particular, if
$F:\Md(\CG)\to\CA$ is an exact functor into an abelian category,
which annihilates all torsion $\CG$-modules, then $F$ factors
uniquely through $Q$.

In this paper we characterize a functor which is a generalized lax
epimorphism, with the help of a conditioned right cancellation of
a functor constructed in a canonical way from the initial one; see
Theorem \ref{resff}. As consequences we deduce in Corollary
\ref{abloc} a characterization of functors inducing an abelian
localization, and in Corollary \ref{relative} an additive version
of the ``Lemme de comparaison'' (see \cite[Theorem 4.1]{SGA4}),
giving a necessary and sufficient condition for a morphism of
rings with several objects to induce an equivalence at the level
of two localizations of the respective module categories. Note
that we shall call conditioned epimorphism a functor satisfying
the above mentioned conditional cancellation property. First we
study such functors in Section \ref{gcface}, the main result being
Theorem \ref{condepi}. We give also applications of our
characterizations for some more or less classical cases. Thus we
deduce the classical results concerning of (flat) epimorphisms of
unitary rings (Proposition \ref{ringepi} and Corollary
\ref{flatep}), but also the main result of Krause's paper
\cite{KE}, concerning epimorphisms up to direct factors
(Proposition \ref{epiupf}). Another characterization of a functor
which induces an abelian localization as in Corollary \ref{abloc}
is the subject of \cite{L}. Inspired by this approach we found in
Proposition \ref{gf} some sufficient conditions for a functor to
be a generalized lax epimorphism. In addition we discuss an
example (Example \ref{exl}), where we clarify  a point which is
called ``obscure'' in \cite{L}.

\section{Generalized closed functors and conditioned epimorphisms}\label{gcface}

We fix in this Section the notations as follows: $\CG$ is a ring
with several objects, $\CL$ is a localizing subcategory of
$\Md(\CG)$, $\CC=\Md(\CG)/\CL$ is the corresponding quotient
category, with the quotient functor $Q:\Md(\CG)\to\CC$, having the
right adjoint $R:\CC\to\Md(\CG)$.  We consider also a morphism of
rings with several objects $S:\CU\to\CG$.

We call {\em generalized $\CL$-closed} a functor $F:\CG\to\CA$,
into a cocomplete, abelian category $\CA$, provided that the
induced functor $F^*:\Md(\CG)\to\CA$ annihilates all torsion
modules (that is $\CL\subseteq\Ker F^*$) and $F^*$ preserves
exactness of sequences of the form $0\to M\to N\to L\to 0$ with
$L\in\CL$. About the morphism of rings with several objects $S$ we
say that it is a {\em $\CL$-conditioned epimorphism} if the
equality $F\circ S=F'\circ S$ implies $F=F'$, provided that the
supplementary condition $F$ is generalized $\CL$-closed holds
true. Remark that an this implication, without any supplementary
condition, means precisely that $S$ is an epimorphism in the
category of rings with several objects (see Lemma \ref{small}
bellow). In the next proposition we characterize those functors
which are generalized $\CL$-closed. Note first:

\begin{remark}\label{term} A module $X\in\Md(\CG)$ is
$\CL$-closed, in the classical sense, if and only if the functor
$X\opp:\CG\to\Ab\opp$ is generalized $\CL$-closed, explaining our
terminology. Indeed, it is enough to observe that the induced
functor is given by
\[(X\opp)^*=\Hom_\CG(-,X):\Md(\CG)\to\Ab\opp.\]
\end{remark}

\begin{proposition}\label{gencl}
The following are equivalent for a functor $F:\CG\to\CA$ into a
cocomplete, abelian category $\CA$:
\begin{itemize}
\item[{\rm (i)}] The functor $F:\CG\to\CA$ is generalized
$\CL$-closed. \item[{\rm (ii)}] $F_*A$ is $\CL$-closed for all
$A\in\CA$, or equivalently there exists $F_\star:\CA\to\CC$ such
that $F_*\cong R\circ F_\star$. \item[{\rm (iii)}] $F^*$ factors
trough $Q$ i.e. there exists $F^\star:\CC\to\CA$ such that
$F^*\cong F^\star\circ Q$.
\end{itemize}
Moreover if these conditions are satisfied, then $F^\star$ is the
left adjoint of $F_\star$.
\end{proposition}

\begin{proof} (i)$\Rightarrow$(ii). Let $A\in\CA$ and $L\in\CL$. The isomorphism
\[\Hom_\CG(L,F_*A)\cong\CA(F^*L,A)=\CA(0,A)=0\] shows that
$F_*A$ is $\CL$-torsion free. Further we consider a short exact
sequence $0\to M\to N\to L\to 0$, with $N$ projective and
$L\in\CL$. By assumption we have $F^*M\cong F^*N$, so
\[\Hom_\CG(M,F_*A)\cong\CA(F^*M,A)\cong\CA(F^*N,A)\cong\Hom_\CG(N,F_*A).\]
Using this together with the exact sequence of abelian groups
\begin{align*}
0=\Hom_\CG(L,F_*A)&\to\Hom_\CG(N,F_*A)\to\Hom_\CG(M,F_*A)\\
&\to\Ext^1_\CG(L,F_*A)\to\Ext^1_\CG(N,F_*A)=0,
\end{align*}
we deduce $\Ext^1_\CG(L,F_*A)=0$, thus $F_*A$ is $\CL$-closed.
Since $R$ is fully faithful, this is property is equivalent to the
factorization of $F_*$ through $R$.

(ii)$\Rightarrow$(iii). First we shall show that $F^*L=0$ for all
$L\in\CL$. Indeed for all $A\in\CA$ the isomorphism
\[\CA(F^*L,A)\cong\Hom_\CG(L,F_*A)=0\] proves our claim.
Let now $\alpha:M\to N$ be a $\CG$-linear map such that $Q\alpha$
is an isomorphism. In particular, the cokernel of this map belongs
to $\CL$, so $F^*\alpha$ is an epimorphism, since $F^*$ is right
exact. Moreover, for all $A\in\CA$ we have the isomorphisms (in
the category of abelian group homomorphisms):
\[\CA(F^*\alpha,A)\cong\Hom_\CG(\alpha,F_*A)\cong\Hom_\CG(\alpha,(R\circ
F_\star)A)\cong\CC(Q\alpha,F_\star A),\] showing that
$\CA(F^*\alpha,A)$ is bijective, therefore $F^*\alpha$ is a split
monomorphism. Thus $F^*\alpha$ is an isomorphism, so $F^*$ factors
through $Q$.

(iii)$\Rightarrow$(i) is obvious.

Using the fully faithfulness of $R$, we have the following natural
isomorphisms, for all $A\in\CA$ and all $C\in\CC$:
\[\CC(C,F_\star A)\cong\Hom_\CG(RC,(R\circ
F_\star)A)\cong\Hom_\CG(RC,F_*A)\\
\cong\CA((F^*\circ R)C,A),\] showing that $F^*\circ R\cong
F^\star\circ Q\circ R\cong F^\star$ is the left adjoint of
$F_\star$.
\end{proof}

In the sequel we want to characterize the conditioned epimorphisms
of rings with several objects. For an easier reference we recall
the following:

\begin{lemma}\label{lradj}\cite[Section 4]{KF} With the above notations, the
restriction functor
\[S_*:\Md(\CG)\to\Md(\CU),\ S_*X=X\circ S\hbox{ for all
}X\in\Md(\CG)\] has not only a left adjoint, namely the induction
functor, which is determined uniquely up to a natural isomorphism
by
\[S^*:\Md(\CU)\to\Md(\CG),\ S^*U=SU\ (U\in\CU)\hbox{ and }S^*\hbox{ is
colimits preserving},\] but also a right adjoint, respectively the
functor
\[^*S:\Md(\CU)\to\Md(\CG),\ (^*SX)G=\Hom_\CU(\CG(S-,G),X).\] Consequently, $S_*$ is
exact and preserves limits and colimits. \end{lemma}

Note that the restriction and the induction functor from the
preceding Lemma agree with those defined in Introduction, after
the identification of a ring with several objects with its image
in the module category over that ring, via the Yoneda embedding.

\begin{lemma}\label{sisff} If $S$ is surjective on objects, then the
restriction functor $S_*$ is faithful and reflects isomorphisms.
\end{lemma}

\begin{proof}
Since $S$ is surjective on objects, it follows that $X\circ S=0$
implies $X=0$ for all $X\in\Md(\CG)$, what means $S_*$ reflects
zero objects. By Lemma \ref{lradj}, the functor $S_*$ is exact
therefore it commutes with images. But such a functor (exact and
reflecting zero objects) is faithful and reflects isomorphisms.
\end{proof}

\begin{lemma}\label{qss} Suppose that $G$ is a $\CL$-closed
module, for all $G\in\CG$ and $S$ is surjective on objects. Then
the following are equivalent:
\begin{itemize}
\item[{\rm (i)}] $S_*\circ R$ is fully faithful. \item[{\rm (ii)}]
$Q\circ S^*\circ S_*\cong Q$ naturally. \item[{\rm (iii)}]
$(Q\circ S^*\circ S_*)G\cong QG$ naturally, for all $G\in\CG$.
\end{itemize}
\end{lemma}

\begin{proof} According to Lemma \ref{sisff}, $S_*$ is
faithful, hence the arrow of adjunction
\[\mu_X:(S^*\circ S_*)X\to X\] is an epimorphism, for all
$X\in\Md(\CG)$. The arrow of the adjunction between $Q\circ S^*$
and $S_*\circ R$ is given by
\[Q\mu_{RC}:(Q\circ S^*)\circ(S_*\circ R)C\to C\hbox{ for all
}C\in\CC.\] Clearly $S_*\circ R$ is fully faithful, exactly if
$Q\mu_{RC}$ is an isomorphism for all $C\in\CC$, or equivalently
$\mu_X$ has torsion kernel for all $\CL$-closed $X\in\Md(\CG)$. On
the other hand (ii) is equivalent to the fact $\ker\mu_X\in\CL$
for all $X\in\Md(\CG)$, therefore (ii)$\Leftrightarrow$(i)
follows.

(i)$\Rightarrow$(iii) As we have seen, $\mu_X\in\CL$ for all
$\CL$-closed $X\in\Md(\CG)$. In particular $\ker\mu_G\in\CL$ for
all $G\in\CG$. Applying the exact functor $Q$ to the exact
sequence
\[0\to\ker\mu_G\to(S^*\circ S_*)G\overset{\mu_G}\to G\to0,\] we
obtain the desired isomorphism.

(iii)$\Rightarrow$(ii) For an arbitrary module $X\in\Md(\CG)$,
there is an exact sequence in $\Md(\CG)$
\[0\to Y\to\bigoplus G_i\to X\to 0.\] We apply the colimits
preserving functor $S^*\circ S_*$ (see Lemma \ref{lradj}), and the
Ker-Coker lemma for the obtained diagram shows that $\ker\mu_X$ is
a quotient of the $\CL$-torsion module $\bigoplus\ker\mu_{G_i}$,
therefore it is also $\CL$-torsion.
\end{proof}

\begin{lemma}\label{cepi} Suppose that $G$ is a $\CL$-closed
module, for all $G\in\CG$ and $S$ is surjective on objects. If
$S_*\circ R$ is fully faithful, then $S$ is a $\CL$-conditioned
epimorphism.
\end{lemma}

\begin{proof}
Let $F,F':\CG\to\CA$ two functors into a cocomplete, abelian
category, such that $F$ is generalized $\CL$-closed, and $F\circ
S=F'\circ S$. Then we obtain in turn the following natural
isomorphisms: $S_*\circ F_*\cong S_*\circ F'_*$, so $Q\circ
S^*\circ S_*\circ F_*\cong Q\circ S^*\circ S_*\circ F'_*$ and
$Q\circ F_*\cong Q\circ F'_*$ by Lemma \ref{qss}; further $R\circ
Q\circ F_*\cong R\circ Q\circ F'_*$, so $F_*\cong R\circ Q\circ
F'_*$, since $F$ is generalized $\CL$-closed, equivalently $F_*$
factors through $R$ by Proposition \ref{gencl}. From the arrow of
adjunction $\id_{\Md(\CG)}\to R\circ Q$, we obtain a natural
morphism \[F'_*A\to(R\circ Q\circ F'_*)A\cong F_*A\hbox{ for all
}A\in\CA,\] which induces the isomorphism $(S_*\circ
F'_*)A\overset{\cong}\to(S_*\circ F_*)A$. Because $S_*$ reflects
isomorphisms, we deduce that the functors $F_*$ and $F'_*$ are
naturally isomorphic. Therefore $F\cong F'$ naturally. But $F$ and
$F'$ coincide on objects, $S$ being surjective on objects. Thus
$F=F'$.
\end{proof}

\begin{theorem}\label{condepi} If $S:\CU\to\CG$ is bijective on
objects and $G$ is a $\CL$-closed module, for all $G\in\CG$, then
the following statements are equivalent:
\begin{itemize} \item[{\rm (i)}] $S_*\circ R$ is full.
\item[{\rm (ii)}] $S_*\circ R$ is fully faithful. \item[{\rm
(iii)}] $S$ is a $\CL$-conditioned epimorphism.
\end{itemize}
\end{theorem}

\begin{proof} (i)$\Leftrightarrow$(ii) is immediate, since both $S_*$
and $R$ are known to be faithful and (ii)$\Rightarrow$(iii)
follows by Lemma \ref{cepi}.

For the implication (iii)$\Rightarrow$(i), we have to show that
the abelian group homomorphism
\[\Hom_\CG(X,Y)\to\Hom_\CU(S_*X,S_*Y)\]
induced by $S_*$ is surjective for all $\CL$-closed
$X,Y\in\Md(\CG)$. In order to do this, we use the argument of
\cite[Lemma 5]{KE}, observing in addition that the functor
$F:\CG\to\Ab\opp$, given by $F=X\opp\oplus Y\opp$ is generalized
$\CL$-closed.
\end{proof}

\section{When the restriction functor is fully
faithful}\label{rfff}

Let $T:\CU\to\CC$ be any (additive) functor, where $\CU$ and $\CC$
are two arbitrary (preadditive) categories. Following \cite{KE},
the functor $T$ has a canonical factorization $T=I\circ S$, where
$S:\CU\to\CG$ is bijective on objects and $I:\CG\to\CC$ is fully
faithful. Moreover, this factorization is unique up to an
isomorphism of categories. Actually the objects of $\CG$ are the
same as the objects of $\CU$ and $\CG(U',U)=\CC(TU',TU)$, for all
$U',U\in\CU$. The functor $S$ is the identity on objects and
$Su=Tu$ for all maps $u:U'\to U$ in $\CU$. The functor $I$ is the
identity on maps and $IU=TU$, for all $U\in\CU$ (see \cite[Lemma
1]{KE}). Observe that, if $T(\CU)$ is the full subcategory of
$\CC$ consisting of those objects of the form $T(U)$ with
$U\in\CU$, then the categories $\CG$ and $T(\CU)$ are equivalent.
Indeed, if $\CU\overset{T'}\to T(\CU)\overset{T''}\to\CC$ is the
factorization of $T$ through its image, then
$\CU\overset{S}\to\CG\overset{I''}\to T(\CU)$ is the canonical
factorization of $T'$, where $I''$ is the identity on maps and
$I''U=T'U$ for all $U\in\CU$.  By construction $T'$ is surjective
on objects, so we deduce that $I''$ is an equivalence.

Assume now that the category $\CC$ is abelian, AB5. The canonical
factorization $\CU\overset{S}\to\CG\overset{I}\to\CC$ of $T$
induces a diagram of categories and functors \[\diagram
\Md(\CU)\rrto<0.5mm>^{T^*}\drto<0.5mm>^{S^*}&&{\hskip2mm\CC\hskip5mm}\llto<0.5mm>^{T_*}\dlto<0.5mm>^{I_*}\\
&\Md(\CG)\ulto<0.5mm>^{S_*}\urto<0.5mm>^{I^*}&\\
\enddiagram\]
in which we have obviously $T^*\cong I^*\circ S^*$ and $T_*\cong
S_*\circ I_*$ naturally.

In this Section we consider a functor $T:\CU\to\CC$ defined on a
ring with several objects with values in an abelian, AB5 category
$\CC$, together with its canonical factorization
$\CU\overset{S}\to\CG\overset{I}\to\CC$. Consider also the adjoint
pair $(T^*,T_*)$ induced by $T$.

\begin{lemma}\label{tf} With the above notations the following are
equivalent:
\begin{itemize}
\item[{\rm(i)}] The functor $T_*$ is faithful. \item[{\rm(ii)}]
The functor $I$ identifies $\CG$ with a generating, small
subcategory of $\CC$. \item[{\rm(iii)}] $T(\CU)$ is a generating
subcategory of $\CC$.
\end{itemize}
Moreover if one, since all, of these conditions holds, then the
category $\CC$ is Grothendieck, and the adjoint pair $(I^*,I_*)$
is a localization. Consequently $\Ker I^*$ is a localizing
subcategory of $\Md(\CG)$.
\end{lemma}

\begin{proof}
(i)$\Rightarrow$(ii). The functor $I$ is fully faithful by
construction, so it identifies $\CG$ with a (small) full
subcategory of $\CC$. Let $\gamma$ be a map in $\CC$ such that
$I_*\gamma=0$. Then $T_*\gamma=(S_*\circ I_*)\gamma=0$, so
$\gamma=0$ since $T_*$ is faithful. It follows that $I_*$ is
faithful, meaning precisely that $\CG$ is a small generating
subcategory of $\CC$. Therefore $\CC$ is Grothendieck, and
$(I^*,I_*)$ is a localization by Gabriel--Popescu theorem.

(ii)$\Rightarrow$(i). The condition (ii) is equivalent to the fact
that $I_*$ is faithful. But $S_*$ is also faithful, by Lemma
\ref{sisff}, so the same is true for $T_*\cong S_*\circ I_*$.

The equivalence (ii)$\Leftrightarrow$(iii) follows by the above
observation, that the categories $\CG$ and $T(\CU)$ are
equivalent.
\end{proof}

Now we are in position to prove the main result of this work:

\begin{theorem}\label{resff} The functor $T$ is a generalized lax
epimorphism if and only if the following conditions hold true:
\begin{itemize}
\item[{\rm (1)}] $\CG$ generates $\CC$; consequently $\Ker I^*$ is
a localizing subcategory of $\Md(\CG)$, and $\CC$ is a
Grothendieck category. \item[{\rm (2)}] $S$ is a $\Ker
I^*$-conditioned epimorphism.
\end{itemize}
\end{theorem}

\begin{proof} Provided that $\CG$ generates $\CC$ (therefore $\Ker I^*$ is
a localizing subcategory of $\Md(\CG)$), we shall show that every
$G\in\CG$ is a $\Ker I^*$-closed $\CG$-module, in order to verify
the hypotheses of Theorem \ref{condepi}. But $I$ is fully faithful
by construction, thus we have the isomorphisms:
\[\CG(-,G)\cong\CC(I-,IG)=I_*(IG)=(I_*\circ I^*)G,\] for every $G\in\CG$, proving our claim.
Now we have only to combine Theorem \ref{condepi} and Lemma
\ref{tf}.
\end{proof}

Provided that $\CC$ is a Grothendieck category, we say that
$T:\CU\to\CC$ satisfies the {\em Ulmer's criterion of flatness},
if for every (finite) set of morphisms $u_i:U_i\to U$ in $\CU$,
with $1\leq i\leq n$, there are objects $V_j\in\CU$, with $j\in
J$, and morphisms $u_{ij}:V_j\to U_i$, with $i\in\{1,\ldots,n\}$
and $j\in J$, such that for each $j$ we have
$\sum_{i=1}^nu_iu_{ij}=0$ and the sequence
\[\bigoplus_{j\in
J}TV_j\overset{(Tu_{ij})}\longrightarrow\bigoplus_{i=1}^n
TU_i\overset{(Tu_i)}\longrightarrow TU\] is exact. By
\cite[Theorem]{U} we learned that the induced functor
\[T^*:\Md(\CU)\to\CC\] is exact if and only if $T$ satisfies the
Ulmer's criterion of flatness.

For a morphism $u:V\to U$ in $\CU$ and a submodule $X\leq U$ we
denote by $(X:u)\leq V$ the inverse image of $X$ through $u$.
Recall from \cite{GG} that a {\em(right) Gabriel filter} on a a
ring with several objects $\CU$ is a family $\FF=\{\FF_U\mid
U\in\CU\}$, where each $\FF_U$ is a set of subobjects of $U$
satisfying:
\begin{itemize}
\item[{\rm GF1.}] $U\in\FF_U$ for all $U\in\CU$. \item[{\rm GF2.}]
For every morphism $u:V\to U$ in $\CU$ and every $X\in\FF_U$ it
holds $(X:u)\in\FF_V$. \item[{\rm GF3.}] If $U\in\CU$, then a
submodule $X\leq U$ belongs to $\FF_U$, whenever there exists
$Y\in\FF_U$ with the property $(X:u)\in\FF_V$ for any morphism
$u:V\to U$ with $\ima u\leq Y$.
\end{itemize}
We know that, for every $U\in\CU$, $\FF_U$ is a filter on the
lattice of submodules of $U$ (that is $X,Y\in\FF_U\Rightarrow
X\cap Y\in\FF_U$ and $X\in\FF_U,Y\leq U,X\leq Y\Rightarrow
Y\in\FF_U$). Moreover there is a bijection between localizing
subcategories of $\Md(\CU)$ and Gabriel filters on $\CU$, given by
$\CL\mapsto\FF(\CL)$ for any localizing subcategory $\CL$ of
$\Md(\CU)$, where:
\[\FF(\CL)_U=\{X\leq U\mid U/X\in\CL\},\hbox{ for all }U\in\CU.\]
(For details concerning Gabriel filters on rings with several
objects see \cite[Section 2.1]{GG}).

\begin{corollary}\label{abloc} The adjoint pair
$(T^*,T_*)$ induced by $T$ is an abelian localization if and only
if the following conditions hold:
\begin{itemize} \item[{\rm (1)}] $\CG$ generates
$\CC$; consequently $\CC$ is a Grothendieck category. \item[{\rm
(2)}] $S$ is a $\Ker I^*$-conditioned epimorphism. \item[{\rm
(3)}] $T$ satisfies the Ulmer's criterion of flatness.
\end{itemize}
Moreover if these conditions are satisfied, then $\CC$ is the
quotient of $\Md(\CU)$ modulo the localizing subcategory
corresponding to the Gabriel filter $\FF$ in $\CU$, where
\[\FF_U=\{X\leq U\mid T^*X\cong TU\hbox{ naturally}\},\]
for all $U\in\CU$.
\end{corollary}

\begin{proof} The necessity and sufficiency of conditions (1), (2) and (3)
in order to derive that $T$ induces an abelian localization is an
immediate consequence of Theorem \ref{condepi} combined with the
Ulmer's criterion of flatness. For the last remaining statement,
observe that $\CC$ is equivalent to $\Md(\CU)/\Ker T^*$, provided
that $(T^*,T_*)$ is a localization. But, for every submodule
$X\leq U$ we have $U/X\in\Ker T^*$ exactly if $T^*X\cong TU$.
\end{proof}

\begin{corollary}\label{relative} Let $P:\CU\to\CU'$ be a morphism of rings with several objects,
and let $\CL'$ be a localizing subcategory of $\Md(\CU')$. We
consider the canonical factorization
$\CU\overset{S}\to\CG\overset{I}\to\Md(\CU')/\CL'$ of the functor
$T=T_{P,\CL'}:\CU\to\Md(\CU')/\CL'$ given by $TU=Q'(PU)$, for all
$U\in\CU$, where $Q':\Md(\CU')\to\Md(\CU')/\CL'$ denotes the
quotient functor. Then the functor $P$ induces an equivalence
$\Md(\CU)/\CL\to\Md(\CU')/\CL'$, for some localizing subcategory
$\CL$ of $\Md(\CU)$, if and only if $\CG$ generates
$\Md(\CU')/\CL'$, $S$ is a $\Ker I^*$-conditioned epimorphism and
$T$ satisfies the Ulmer's criterion of flatness. If this is the
case, then we have also
\[\CL=\{X\in\Md(\CU)\mid P^*X\in\CL'\}.\]
\end{corollary}

\begin{proof} Denoting $\CL=\Ker T^*$, the functor
$P$ induces an equivalence of categories as stated if and only if
$T$ induces an abelian localization, therefore Corollary
\ref{abloc} applies. Moreover if this the case,
\[\Ker T^*=\{X\in\Md(\CU)\mid P^*X\in\CL'\}.\]
\end{proof}

\begin{remark}\label{old} {\rm Corollary \ref{relative} gives necessary and sufficient
conditions for a morphism of two rings with several objects $\CU$
and $\CU'$ to induce an equivalence at the level of two
localizations of $\Md(\CU)$ and $\Md(\CU')$, respectively. In this
sense it is an additive version of \cite[Theorem 4.1]{SGA4} (see
also \cite[Corollary 4.5]{L}). But it also gives a partial answer
to a question occurring naturally in \cite{CM}: Given two
Grothendieck categories, $\CA$ and $\CB$, a pair of adjoint
functors between them $R:\CA\to\CB$ at the right and $L:\CB\to\CA$
at the left, and a hereditary torsion class $\CT$ in $\CA$, what
additional hypotheses should be considered, such that
$\{B\in\CB\mid LB\in\CT\}$ is a hereditary torsion class?}
\end{remark}

\section{Ordinary epimorphisms of rings with several objects}

In this section we shall see how do our result generalize the
classical case of (flat) epimorphisms of rings (see \cite{M} or
\cite{P}). Observe first that, for the localizing subcategory
$\CL=0$ of a module category $\Md(\CG)$ over a ring with several
objects $\CG$, every $\CG$-module is $0$-closed, and every functor
$F:\Md(\CG)\to\CA$ into a cocomplete, abelian category $\CA$ is
generalized $0$-closed. We shall say that the ring with several
objects $\CG'$ has less objects than the ring with several objects
$\CG$ if the cardinality of isomorphism classes of objects in
$\CG'$ is smaller than the one of objects in $\CG$.

\begin{lemma}\label{small} Consider a morphism
of rings with several objects $S:\CU\to\CG$, which is surjective
on objects. The following are equivalent:
\begin{itemize}
\item[{\rm (i)}] $S$ is a $0$-conditioned epimorphism. \item[{\rm
(ii)}] $S$ is an epimorphism in the category of rings with several
objects. \item[{\rm (iii)}] For every two morphisms of rings with
several objects $F,F':\CG\to\CG'$, where $\CG'$ has less objects
than $\CG$, we have $F\circ S=F'\circ S$ implies $F=F'$.
\end{itemize}
\end{lemma}

\begin{proof} (i)$\Rightarrow$(ii) and (ii)$\Rightarrow$(iii) are
obvious.

(iii)$\Rightarrow$(i). Let $F,F':\CG\to\CA$ be two (arbitrary)
functors into a cocomplete, abelian category such that $F\circ
S=F'\circ S$. Since $S$ is surjective in objects, it follows that
$F$ and $F'$ coincide on objects. If we consider
$\CG'=F(\CG)=F'(\CG)$ (considered as a full subcategory of $\CA$),
then $\CG'$ has less objects than $\CG$. It follows $F=F'$ by
applying (iii) to factorizations through image of $F$ and $F'$.
\end{proof}

An immediate consequence of Theorem \ref{resff} and Lemma
\ref{small} is then the following well--known characterization of
epimorphisms of unitary rings:

\begin{proposition}\label{ringepi} Let $A$ and
$B$ be two unitary rings, and let $\varphi:A\to B$ a unitary ring
homomorphism. Then $\varphi$ is an epimorphism in the category of
unitary rings if and only if the restriction functor
\[\varphi_*:\Md(B)\to\Md(A),\ \varphi_*Y=Y\]
is fully faithful.
\end{proposition}

From Corollary \ref{abloc} follows as well the case of flat
epimorphisms of rings:

\begin{corollary}\label{flatep} With the notations made in Proposition
\ref{ringepi}, consider the adjoint pair $(\varphi^*,\varphi_*)$,
where
\[\varphi^*:\Md(A)\to\Md(B),\ \varphi^*X=X\otimes_AB\] is the
induction functor and $\varphi_*$ is the restriction functor
defined above. Then this adjoint pair is a localization if and
only if $\varphi$ is a flat epimorphism of rings (i.e. is an
epimorphism of unitary rings making $B$ into a flat $A$-module).
\end{corollary}

Another interesting result concerning lax epimorphisms of rings
with several objects makes the object of investigations of
Krause's work \cite{KE}. In order to derive this result from our
Theorem \ref{resff}, we need the following:

\begin{lemma}\label{indeq}
Let $I:\CG\to\CV$ be a morphism of rings with several objects,
inducing a localization $(I^*,I_*)$. Then $I^*$ and $I_*$ are
mutually inverse equivalences of categories if and only if $G$ is
a $\Ker I^*$-closed $\CG$-module, for all $G\in\CG$.
\end{lemma}

\begin{proof} The direct implication is obvious since, if $I^*$ is
an equivalence, then $\Ker I^*$ is $0$.

Conversely, let $X\in\Md(\CG)$ arbitrary. We want to show that the
arrow of adjunction $X\to(I_*\circ I^*)X$ is an isomorphism. In
order to do this, apply the colimit preserving functor $I_*\circ
I^*$ (see Lemma \ref{lradj}) to a free presentation of $X$. We
obtain a diagram with exact rows: \[\diagram \bigoplus
G_j'\rto\dto&\bigoplus G_i\rto\dto&X\rto\dto&0\\
\bigoplus(I_*\circ I^*)G_j'\rto&\bigoplus(I_*\circ
I^*)G_i\rto&(I_*\circ I^*)X\rto&0
\enddiagram\]
The first two vertical morphisms are isomorphisms by hypothesis,
therefore the same is true for the third.
\end{proof}

\begin{proposition}\label{epiupf} Let
$T:\CU\to\CV$ be a morphism of rings with several objects, and let
$\CU\overset{S}\to\CG\overset{I}\to\CV$ be its canonical
factorization. Then the following are equivalent:
\begin{itemize}
\item[{\rm (i)}] The functor $T$ is a lax epimorphism. \item[{\rm
(ii)}] $S$ is an epimorphism in the category of rings with several
objects, and $I$ induces an equivalence $\Md(\CG)\to\Md(\CV)$.
\item[{\rm (iii)}] $S$ is an epimorphism in the category of rings
with several objects, and for every object $V\in\CV$, there exists
a finite set of objects $G_i\in\CG$ with maps $v_i:V\to IG_i\to V$
in $\CV$ such that $1_V=\sum_{i}v_i$.
\end{itemize}
\end{proposition}

\begin{proof} (i)$\Leftrightarrow$(ii). We shown in the
proof of Theorem \ref{resff} that $G$ is $\Ker I^*$-closed for all
$G\in\CG$. Then the equivalence follows from Theorem \ref{resff},
Lemma \ref{small} and Lemma \ref{indeq}.

The equivalence (ii)$\Leftrightarrow$(iii) follows by \cite[Lemma
4]{KE}
\end{proof}

Note that a morphism of rings with several objects $T:\CU\to\CV$
satisfying the condition (iii) (therefore all) in Proposition
\ref{epiupf} above, is called an {\em epimorphism up to direct
factors} in \cite{KE}. Thus in Proposition \ref{epiupf} we give
another proof of the main result in \cite{KE} that an epimorphism
up to direct factors is exactly what we call a lax epimorphism.

\section{A particular case and an example}\label{partex}

We reset the notations and assumptions made in Section \ref{rfff},
namely $T:\CU\to\CC$ is a functor defined on a ring with several
objects, with values into an abelian AB5 category, and
$\CU\overset{S}\to\CG\overset{I}\to\CC$ is its canonical
factorization. To characterize the situation in which $T$ induces
an abelian localization $(T^*,T_*)$, as in Corollary \ref{abloc},
is the object of investigations in \cite{L}. Inspired by this, we
may find some sufficient conditions for $T_*$ being fully
faithful, as may be seen in the following:

\begin{proposition}\label{gf} Consider the following conditions relative to $T$:
\begin{itemize}
\item[{\rm (G)}] $T(\CU)$ generates $\CC$; consequently $\CC$ is
Grothendieck. \item[{\rm (F)}] If $\gamma:T(U)\to T(U')$ is a map
in $\CC$, where $U,U'\in\CU$, then there are objects $V_j\in\CU$
and maps $u_j:V_j\to U$ and $u'_j:V_j\to U'$, with $j\in J$, such
that $\gamma Tu_j=Tu'_j$ for all $j\in J$, and the sequence
\[\bigoplus TV_j\overset{(Tu_j)}\longrightarrow TU\to0\]
is exact in $\CC$.
\end{itemize}
Then the have: \begin{itemize} \item[{\rm (a)}] If $T$ induces a
localization then (G) and (F) hold true. \item[{\rm (b)}] If (G)
and (F) hold then $T$ is a generalized lax epimorphism.
\end{itemize}
\end{proposition}

\begin{proof} Suppose now that $T$ induces a localization $(T^*,T_*)$. The
condition (G) follows by Lemma \ref{tf} and Corollary \ref{abloc}.
We shall derive the condition (F) by computing $\gamma$ via
calculus of fractions: $\gamma=T^*\alpha(T^*\sigma)^{-1}$, where
$\alpha:X\to U'$, $\sigma:U\to X$ are maps in $\Md(\CU)$, with
$T^*\sigma$ invertible in $\CC$. Chose a presentation
\[\bigoplus V_j\to X\to0\] of $X$ in $\Md(\CU)$, where
$j$ runs over an arbitrary set $J$. For each $j\in J$, compose the
map $V_j\to X$ with $\alpha$, respectively $\sigma$, to obtain
maps $u_j:V_j\to U$ and $u'_j:V_j\to U'$, satisfying the property
$\gamma Tu_j=Tu'_j$. The required exactness of the sequence in (F)
follows by the fact that $T^*\sigma: TU\to T^*X$ is an
isomorphism.

Suppose now that (G) and (F) hold. The condition (G) is equivalent
to $\CG$ generates $\CC$, by Lemma \ref{tf}, so $\Ker I^*$ is a
localizing subcategory of $\Md(\CG)$. In order to apply Theorem
\ref{resff}, we want to show that $S$ is a $\Ker I^*$-conditioned
epimorphism. Let now $\CA$ be a cocomplete, abelian category and
let $F,F':\CG\to\CA$ be two functors, such that $F$ is generalized
$\Ker I^*$-closed, and $F\circ S=F'\circ S$. Then $F$ and $F'$
coincide on objects, since $S$ is bijective on objects. Let
$g:G\to G'$ be a map in $\CG$, and let $U,U'\in\CU$ such that
$SU=G$ and $SU'=G'$. Then $Ig:TU\to TU'$ is a map in $\CC$. By (F)
there exits objects $V_j\in\CU$ and maps $u_j:V_j\to U$ and
$u'_j:V_j\to U'$, with $j\in J$, such that $(Ig)(Tu_j)=Tu'_j$ for
all $j\in J$, and $(Tu_j)_{j\in J}:\bigoplus_{j\in J}TV_j\to TU$
is an epimorphism. Since $I$ is fully faithful, we deduce
$gSu_j=Su'_j$ for all $j\in J$, therefore \[(Fg)\upsilon_j=(F\circ
S)u'_j=(F'\circ S)u'_j=(F'g)\upsilon_j,\] where we denoted
$\upsilon_j=(F\circ S)u_j=(F'\circ S)u_j$, for all $j\in J$. But
the fact that $(Tu_j)_{j\in J}$ is an epimorphism means precisely
that the map $(Su_j)_{j\in J}$ has torsion cokernel. Therefore,
applying the right exact functor, which annihilates all $\Ker
I^*$-torsion module $F^*$, we deduce that $(\upsilon_j)_{j\in J}$
is an epimorphism, therefore $Fg=F'g$, so $F=F'$.
\end{proof}

\begin{remark}
{\rm In \cite[Theorem 1.2]{L} the functor $T$ inducing an abelian
localization is characterized by three conditions, two of which
being (G) and (F) from Proposition \ref{gf} above. The third
condition denoted (FF) in \cite{L} is a particular case of Ulmer's
criterion of flatness.}
\end{remark}

\begin{remark} {\rm We may also observe that in \cite[Theorem 3.7]{L} is given the
Gabriel filter (called there topology) on $\CU$, for which the
category $\CC$ is equivalent to the quotient category of
$\Md(\CU)$ modulo that Gabriel filter (with the terminology of
\cite{L}, $\CC$ is the the category of sheaves over $\CU$
respecting that topology). This filter consists of some submodules
(subfunctors) of free modules $U\in\Md(\CU)$ (representable
functors) which are called there ``epimorphic''. According to
\cite[Lemma 3.4]{L} a submodule $X\leq U$ is an epimorphic
subfunctor of $U$ if and only if $T^*X\cong TU$ naturally, thus we
are lead to the same Gabriel filter as in Corollary \ref{abloc}.}
\end{remark}

\begin{example}\label{exl} {\rm We recall an example in \cite{L},
in order to see how our results give a more comprehensive approach
of phenomena occurring here. Let $(X,\CO_X)$ be ringed space. We
denote by $\Pm(\CO_X)$ and $\Sm(\CO_X)$ the category of
presheaves, respectively sheaves of $\CO_X$-modules. The
sheafication functor $\Pm(\CO_X)\to\Sm(\CO_X)$ is exact and admits
a fully faithful right adjoint, so we are in the situation of a
localization. Further for all open subset $A\subseteq X$, consider
as in \cite[Section 5]{L} the finitely generated projective
presheaf $U_A$ associated to $A$, and denote by $G_A$ the
corresponding sheaf, under the sheafication functor. Then
\[\CU=\{U_A\mid A\hbox{ is a open subset of }X\}\] is a generating
subcategory of $\Pm(\CO_X)$, and
\[\CG=\{G_A\mid A\hbox{ is a open subset of }X\}\] generates
$\Sm(\CO_X)$. Viewing $\CU$ and $\CG$ as full subcategories of
$\Pm(\CO_X)$, respectively $\Sm(\CO_X)$, we know that $\Pm(\CO_X)$
is equivalent to $\Md(\CU)$ and $\Sm(\CO_X)$ is a localization of
$\Md(\CG)$. Denote by $S$ the functor $\CU\to\CG$ given by
$U_A\mapsto G_A$, for all open subsets $A\subseteq X$. The
relation between $\Md(\CU)$ and $\Md(\CG)$ is described in
\cite[Section 5]{L} as ``obscure''. Corollary \ref{abloc}
clarifies this relation, by observing that
\[\CU\overset{S}\to\CG\overset{I}\to\Sm(\CO_X),\] where $I$ is the
inclusion functor, is the canonical factorization of the
restriction at $\CU$ of the sheafication functor. Thus $S$ is a
$\Ker I^*$-conditioned epimorphism.}
\end{example}


\begin{thebibliography}{99}

\bibitem{ABS} J. Ad\'amek, R. El Bashir, M. Sobral, J. Velebil, {On
functors which are lax epimorphisms}, Theory Appl. Categ. 8(2001),
509--521.

\bibitem{SGA4} M. Artin, A. Grothendieck, J.L. Verdier, {Th\'eorie des Topos
et Cohomologie Etale des Sch\'emas}, in {S\'eminaire de
G\'eometrie Alg\`ebrique du Bois--Marie}, Lecture Notes in Math.,
{269}, Springer Verlag, 1972.

\bibitem{G} P. Gabriel, {De cat{\'e}gories abeliennes}, Bull.
Soc. Math. France, {90} (1962) 323--448.

\bibitem{GG} G. Garkusha, {Grothendieck categories}, St. Petesburg
Math. J, {13}(2) (2002), 149--200.

\bibitem{KF} H. Krause, {Functors on locally finitely presented
categories}, Colloq. Math., {75} (1998), 105--131.

\bibitem{KE} H. Krause, {Epimorphisms of additive categories up
to direct factors},  J. Pure Appl. Algebra  {203} (2005),
113--118.

\bibitem{KQ} H. Krause, {Cohomological quotients and smashing
localizations}, Amer. J. Math.  {127} (2005),  no. 6, 1191--1246.

\bibitem{L} W. Lowen, {A generalization for the Gabriel--Popescu
theorem}, J. Pure Appl. Algebra, {190} (2004), 197-211.

\bibitem{MM} S. Mac Lane, I. Moerdijk, {Sheaves in Geometry and Logic.
A first introduction to topos theory}, Universitext,
Springer--Verlag, New York, Berlin, Heidelberg, 1992.

\bibitem{BM} B. Mitchell, {Theory of Categories}, Academic
Press, New York and London, 1965.

\bibitem{M} B. Mitchell, {The dominion of Isbell}, Trans. Amer.
Math. Soc., {167} (1972), 319--331.

\bibitem{CM} C. Modoi, {Equivalences induced by adjoint functors},
Comm. Alg. {31}(5) (2003), 2327--2355.

\bibitem{P} N. Popescu, {Abelian categories with applications to rings and modules},
Academic Press, London, 1973.

\bibitem{PP} N. Popescu, L. Popescu, {Theory of Categories},
Editura Academiei, Bucure\c sti and Sijthoff \& Noordhoff
International Publishers, 1979.

\bibitem{U} F. Ulmer, {A flatness criterion in Grothendieck categories},
Invent. Math., {19} (1973), 331--336.


\end{thebibliography}
\end{document}